\documentclass[12pt]{amsart}
\usepackage{graphicx}
\usepackage{tabularx, verbatim, amsfonts, amsthm, amssymb, color, epsfig, multienum, setspace}
\usepackage{tikz}
\usepackage[pagewise]{lineno}\nolinenumbers
\newtheorem{theorem}{Theorem}[section]
\newtheorem{lemma}[theorem]{Lemma}
\newtheorem{corollary}[theorem]{Corollary}
\theoremstyle{definition}

\newtheorem{definition}[theorem]{Definition}
\newtheorem{example}[theorem]{Example}

\theoremstyle{remark}

\numberwithin{equation}{section}

\newcommand{\kthree}{
\begin{tikzpicture}[scale=0.5]
	\path (5,7.5)		node[rectangle,draw] (g3) {$g_{3}$};
	\path (10,0)		node[rectangle,draw] (g2) {$g_{2}$};
	\path (0,0)	node[rectangle,draw] (g1) {$g_{1}$};

	\draw (g1) -- (g2);
	\draw (g2) -- (g3);
	\draw (g1) -- (g3);

    \node [below] at (5,-0.5){$a$};
    \node [right] at (8,4){$b$};
    \node[left] at (2,4){$c$};
\end{tikzpicture}
}

\newcommand{\kfour}[5]{
\begin{tikzpicture}[scale=#1,auto,main node/.style={circle,draw,font=\sffamily\Large\bfseries}]
	\path (0,0)			node[draw, shape=rectangle] (g4) {$#5$};
	\path (0,7.5)		node[draw, shape=rectangle] (g3) {$#4$};
	\path (9,-5)		node[draw, shape=rectangle] (g2) {$#3$};
	\path (-9,-5)	node (g1)[draw, shape=rectangle] {$#2$};

	\draw (g1) -- (g2);
	\draw (g2) -- (g3);
	\draw (g2) -- (g4);
	\draw (g3) -- (g4);
	\draw (g4) -- (g1);
	\draw (g1) -- (g3);
	\draw (g1) -- (g4);

    \node [right] at (-4.5,1.25){$6$};
    \node [left] at (4.5,1.25){$20$};
    \node[below] at (0,-5){$6$};
    \node[right] at (0,3.75){$8$};
    \node[below] at (4.5,-2.5){$15$};
    \node[below] at (-4.5,-2.5){$6$};

\end{tikzpicture}
}

\renewcommand\qedsymbol{$\blacksquare$}

\newcommand{\N}{\mathbb{N}}
\newcommand{\Z}{\mathbb{Z}}

\begin{document}

\title{Generalized Integer Splines on Arbitrary Graphs}

\author{Lauren L. Rose}
\thanks{Thanks to Tom Zaslavsky for several helpful conversations.}
\address{Department of Mathematics, Bard College, Annandale-on-Hudson, NY, 12504}
\email{rose@bard.edu}

\author{Jeff Suzuki}
\address{Department of Mathematics, Brooklyn College, Brooklyn, NY 11210}
\email{jsuzuki@brooklyn.cuny.edu}



\begin{abstract}
  Generalized integer splines on a graph $G$ with integer edge weights are integer vertex labelings such that if two vertices share an edge in $G$, the vertex labels are congruent modulo the edge weight.  We introduce collapsing operations that reduce any simple graph to a single vertex, carrying with it the edge weight information. This corresponds to a sequence of surjective maps between the associated spline modules, leading to an explicit construction of a $\Z$-module basis in terms of the edge weights.  
\end{abstract}

\maketitle

\section*{Introduction}

Consider a graph $G$ with integer edge weights. Label the vertices so that if two vertices are joined by an edge with weight $w$, then their vertex labels are congruent modulo $w$.  This leads to the following questions:
\begin{itemize}
\item
Given arbitrary edge weights, can we always construct a vertex labeling?  
\item
Can we find a finite set that generates all vertex labelings?
\end{itemize}

In \cite{julia}, Tymozcko, et. al. introduced the study of generalized splines on a graph with edges weighted by ideals in a commutative ring $R$.   Their foundational paper concludes with many open questions, primarily about the $R$-module structure of sets of splines and how much this structure is determined by properties of the ring and/or the underlying graph.  A subsequent survey article \cite{julia2} describes the role that generalized splines play in geometry and topology.  

In this paper we set $R$ to be $\Z$, the ring of integers, although the results in this paper hold when $R$ is an Euclidean Ring, with minor modifications.  We answer both of the above questions affirmatively, and in doing so find a complete characterization of spline modules over $\Z$ for arbitrary graphs.

Generalized splines are a generalization of polynomial splines, which are ubiquitous in applied math, computer graphics and approximation theory.  They are often defined to be piecewise polynomial functions on a polyhedral subdivision of $\mathbb{R}^d$. Alfeld and others \cite{Alfeld} studied the question of finding dimensions and bases for vector spaces of splines of restricted polynomial degree.  In \cite{billera1, billera2}, Billera  pioneered the use of algebraic techniques in the study of polynomial splines.  When the polynomial degrees are unrestricted, the set of all splines on a polyhedral complex will be both a ring and an $R$-module, where $R$ is a polynomial ring in $d$ variables over a field. 

The algebraic study of polynomial splines led to a representation of splines as vertex labels of the dual graph of the polyhedral complex \cite{rose4, rose1}, where the edges are labeled with linear forms.  In this new representation, a set of vertex labels is a spline if each edge label divides the difference between its two vertex labels.  This representation easily extends to arbitrary rings, and in fact arises in geometry and topology.

When $R$ is a principal ideal domain, generalized spline modules are always free of rank $n$,  the number of vertices of the graph. In \cite{smith students}, Handschy, et. al, constructed ``flow-up class'' bases for $n$-cycle graphs with integer edge weights.   In \cite{liu} we proved that flow-up class bases exist for spline modules over arbitrary graphs with integer edge weights, and expanded the class of graphs for which a flow-up class basis could be constructed.

The main result of this paper is the construction of flow-up class bases for the spline module $S_G$, for any arbitrary graph $G$ with integer edge weights.  To do this, we introduce two collapsing operations that transform $G$ into a graph with fewer vertices or edges. Using these operations, we can reduce any connected graph to a point. Then we show that each of these operations induces a surjective $\Z$-module map of the associated spline modules.   The main algebraic result is that $S_G$ is the direct sum of the kernels of these maps.  From there, all that is left to do is to construct a basis for each of the kernels, and show that their pre-images form a basis for $S_G$. A nice feature of this construction is that the first non-zero entry of each basis element, viewed as an $n$-tuple of $\Z^n$, can be written in terms of least common multiples and greatest common divisors of a subset of the edge weights.

%
%

\section{Preliminaries}

We begin with a formal definition of a generalized integer spline. Let $G$  be a graph with $n$ vertices, edge set $E$, and $A : E \rightarrow \N$  an assignment of positive integer weights to edges of $G$.   

We denote the set of splines on $(G, A)$ by $S_{G}(A)$, or just $S_G$ when there is no ambiguity.   

\begin{definition} Let $G$ a graph with $n$ vertices.
A {\textbf{generalized integer spline}} on  $G$ is an $n$-tuple $(g_{1}, g_{2}, \ldots, g_{n}) \in {\Z}^ n$ such that if vertices $v_i$ and $v_j$ are joined by $e \in E$, then $g_{i} \equiv g_{j} \bmod A(e)$.  We refer to these equations as the \textbf{defining equations} for $S_G$.
\end{definition}

\begin{example}  Let $G$ be a 3-cycle with edge labels $a,b,c$ as in Figure\ref{figure3cycle}.  By the definition above, $(g_1,g_2,g_3) \in Z^3$ is a spline in $S_G$ if the following defining equations, one for each edge, are satisfied.
\begin{align*}
 g_1  &\equiv g_2 \bmod {a}\\
 g_2  &\equiv g_3 \bmod {b}\\
 g_3  &\equiv g_1 \bmod {c}
\end{align*}
\end{example}

\begin{figure}
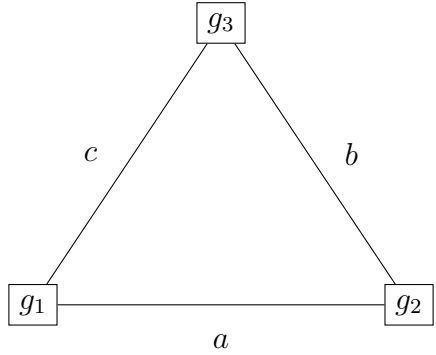


\begin{center}
\kthree
\end{center}
\caption{A $3$-cycle}
\label{figure3cycle}
\end{figure}
{\textbf{Notation:}}  Let  $(r_1, r_2, \ldots r_m)$  denote the GCD and $[r_1, r_2, \ldots r_m]$ denote the LCM of the set $\{r_1, r_2, \ldots r_m\}$. 

The following lemma will be used in later sections.  The proofs are straightforward, so we omit them.

\begin{lemma}  Let $a,b,m,n$ in $\N$.
\label{lemma:gcd}
\begin{enumerate}
\item  $(a,(b,c)) = (a,b,c)$.
\item $[a,[b,c]] = [a,b,c]$.
\item $[(a,c),(b,c)] = ([a,b],c)$.
\item $x \equiv a \bmod m $ and $x \equiv a \bmod n$ if and only if $x \equiv a \bmod [m, n]$.
\end{enumerate}
\end{lemma}

We will need the non-coprime version of the Chinese Remainder Theorem, which we state for convenience.

\begin{theorem}[Non-Coprime Chinese Remainder Theorem] 
\label{theorem:CRT}
Let $a_{i} \in \Z$  and $ m_{i} \in \N$ for $i$ between 1 and $n$. The system of congruences

\begin{align*}
 x  &\equiv a_1 \bmod {m_1}\\
 x  &\equiv a_2 \bmod {m_2}\\
    &\vdots \\
 x  &\equiv a_n \bmod {m_n}
\end{align*}

has a solution if and only if 
$$a_{i} \equiv a_{j} \bmod (m_i, m_j)$$ for all $i, j$ with  $1 \leq i, j \leq n$. 
If a solution exists, then it is unique modulo $[m_1, m_2, \dots , m_n ]$.
\end{theorem}

\section{Collapsing Operations on $G$}


In this section, we define two operations on $(G, A)$ that yield another edge-weighted graph $(G', A')$ with either fewer vertices or the same number of vertices and fewer edges.  We will show that for any $(G, A)$, and vertices $v_i$ and $v_j$, there exists a sequence of operations that yield the weighted edge $(K_{2}, A_{*})$, for some $A_*$ in $\N$.  Moreover, we will show that any spline on  $(K_{2}, A_{*})$ can be extended to a spline on $(G, A)$.  We can then use this collapse sequence to construct a $\Z$-module basis for $S_{G}$. 

\begin{definition}  Let $G$ be a graph and $v$ a vertex of $G$ of degree $d$.  
\begin{enumerate}
\item The {\textbf{star of $v$}}, denoted $st(v)$, is the subgraph consisting of $v$ and all edges that contain $v$.  
\item A {\textbf{$d$-clique}} is a subgraph that is a complete graph on $d$ vertices.
\end{enumerate}
\end{definition}

The first collapse operation replaces the star of $v$ with a $d$-clique, where $d$ is the degree of $v$, using the neighbors of $v$ are the vertices.  All other vertices and edges in the graph remain the same, although the new graph may have multiple edges.  We call this a {\textbf{star-clique}} operation, as we are deleting a vertex star, and then adding new edges to the adjacent $d$ vertices to form a $d$-clique.  This generalizes the standard ${\mathbf{Y-\Delta}}$ transform that takes a 3-star to the 3-clique $K_3$. We now define this process more formally.

\begin{definition} Let $v$ be a vertex of $G$ of degree $d$, with adjacent vertices $\{v_1,\ldots,v_d\}$ and incident edge weights $\{a_1,\ldots,a_d\}$.  The {\textbf{star-clique}} operation transforms $(G,A)$ to $(G_v, A_v)$ as follows.  
\begin{itemize}
\item Remove $v$ and the edges of $st(v)$, but keep the adjacent vertices.
\item Add edges between adjacent vertices to form a $d$-clique.
\item Label the new edge between $v_i$ and $v_j$ with $(a_i,a_j)$.
\end{itemize}
\end{definition}

\begin{example}
Let $v_{2}$ be a vertex of degree $1$, with adjacent vertex $v_{1}$ and incident edge $e$ with weight $a$. The star-clique operation replaces the star of $v_{2}$, an edge, with a 1-clique, which is the single vertex $v_{1}$.  We call this a \textbf{leaf deletion}.  See Figure \ref{fig:leafdelete}.
\end{example}

\begin{figure}
\begin{center}
\begin{tabular}{ccc}
\includegraphics[scale=0.3]{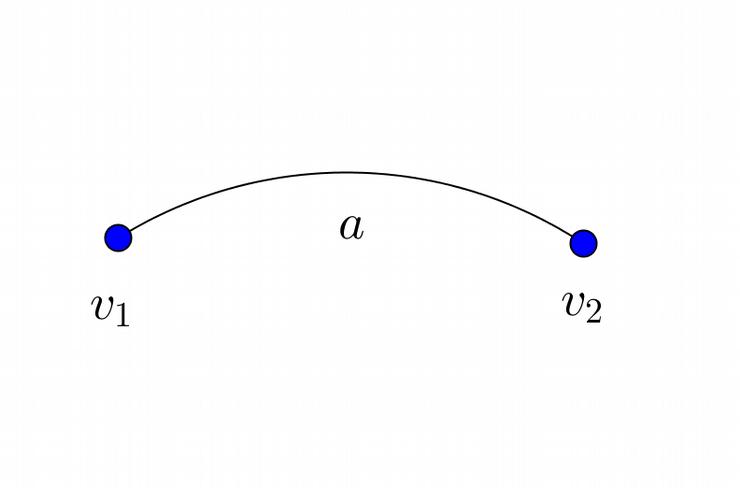}
        &\includegraphics[scale=0.3]{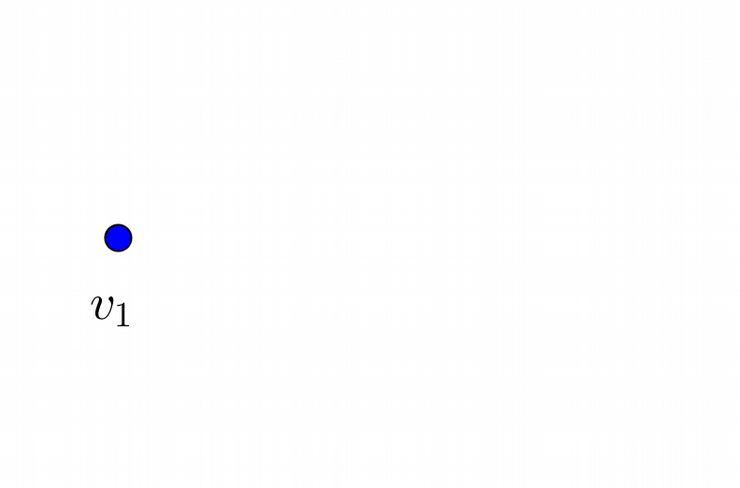}
\end{tabular}
\end{center}
\caption{Leaf Deletion}
\label{fig:leafdelete}
\end{figure}

\begin{example}
Let  $v_{2}$ have degree 2, with adjacent vertices $v_{1}$ and $v_{3}$ and incident edges with weights $a$ and $b$.  The  {\textbf{star-clique}} operation replaces the star of $v_{2}$, i.e. the path $v_{1}v_{2}v_{3}$, with the edge $v_{1}v_{3}$.  This is a {\textbf{path contraction}}, with edge weights $a$ and $b$ replaced by $(a, b)$ (see Figure \ref{fig:pathcontract}).
\end{example}

\begin{figure}
\begin{center}
\begin{tabular}{ccc}
\includegraphics[scale=0.3]{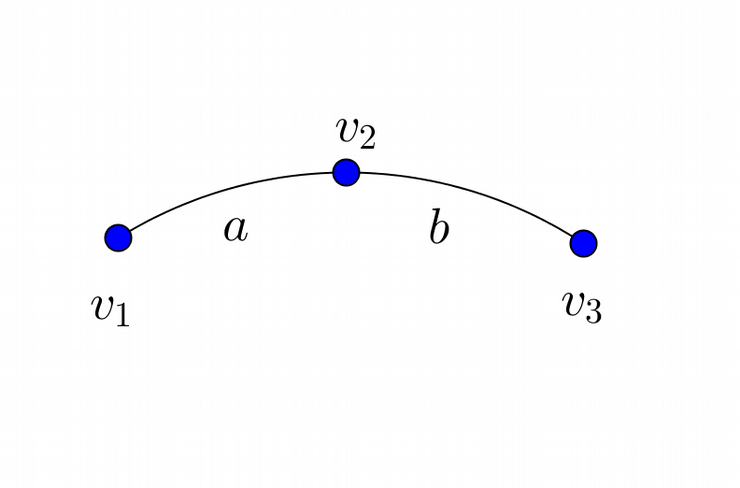}
        &\includegraphics[scale=0.3]{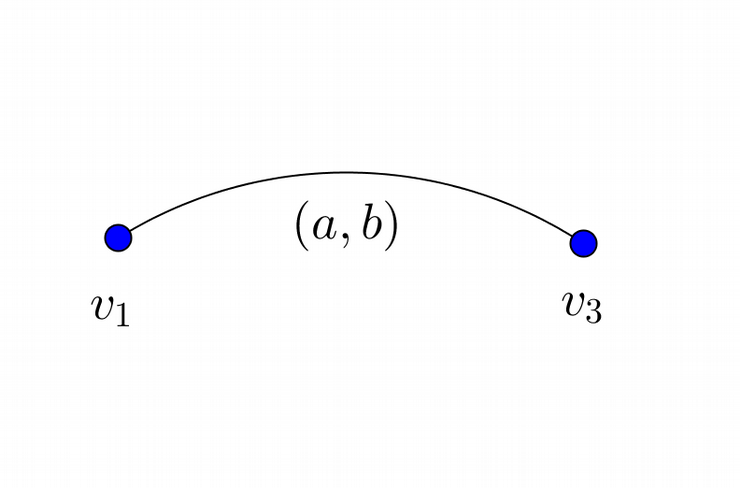}
\end{tabular}
\end{center}
\caption{Path Contraction Operation on $v_2$}
\label{fig:pathcontract}
\end{figure}

\begin{example}
Let  $v_{4}$ be a vertex of degree $3$, with adjacent vertices $v_{1}$, $v_{2}$ and $v_{3}$ and incident edge weights $c,a,b$.  The  {\textbf{star-clique}} operation replaces the star of $v_{4}$, a ``$Y$-graph,'' with a triangle with vertices $v_{1}, v_{2}, v_{3}$.  The new edges will have weights $(a,c),(a,b),(b,c)$, as shown inn Figure \ref{fig:starclique}).
\end{example}

\begin{figure}
\begin{center}
\begin{tabular}{ccc}
\includegraphics[scale=0.3]{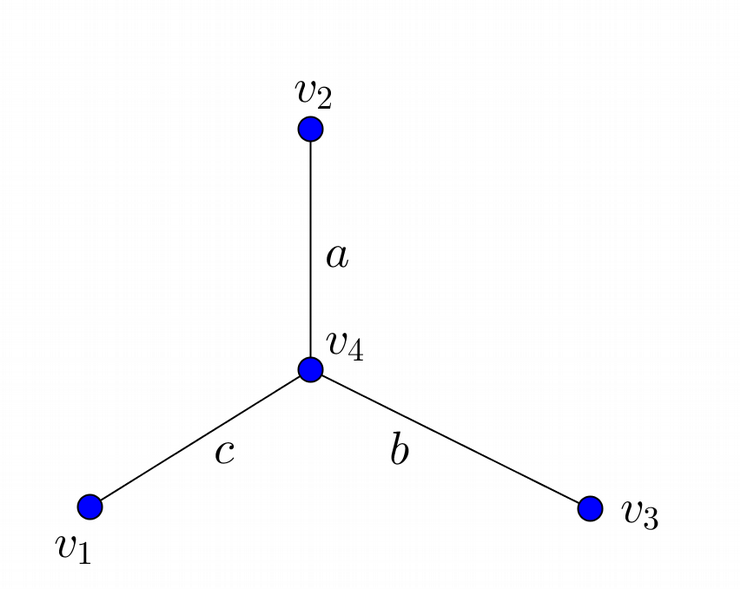}
        &\includegraphics[scale=0.3]{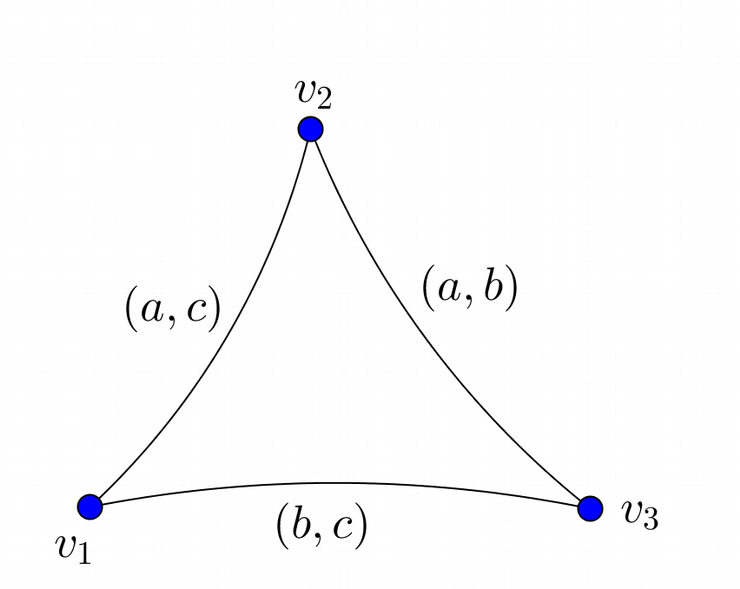}
\end{tabular}
\end{center}
\caption{Star-clique Operation ($Y$-$\Delta$ Transform)}
\label{fig:starclique}
\end{figure}

The new graph $(G_v,A_v)$ has one fewer vertex, but may no longer be a simple graph if some of the $v_i$'s were already be connected in $G$.  

In order to turn $G_v$ into a simple graph, so that we can apply the star-clique operation again, we introduce a second operation, {\textbf{edge collapse}}, that collapses multiple edges between the same pair of vertices (see Figure \ref{fig:edgecollapse}). 

\begin{definition} Let $e_1, e_2, \ldots e_r$ be multiple edges between vertices $v$ and $w$ in $G$ with labels $\{a_1,\ldots,a_r\}$.  The {\textbf{edge-collapes}} operation transforms $(G,A)$ to $(G', A')$ as follows. 
\begin{itemize}
\item Remove $e_1, e_2, \ldots e_r$.
\item  Add a new edge $e$ between $v$ and $w$.
\item  Label $e$ with $[a_1,a_2,\ldots, a_r]$.
\end{itemize}
\end{definition}


\begin{example}
Let $v_{1}, v_{2}$ be two vertices joined by edges of weights $a, b$.  The \textbf{edge-collapse} operation replaces these edges with a single edge of weight $[a, b]$ (see Figure \ref{fig:edgecollapse}).
\end{example}

\begin{figure}
\begin{center}
\begin{tabular}{ccc}
\includegraphics[scale=0.3]{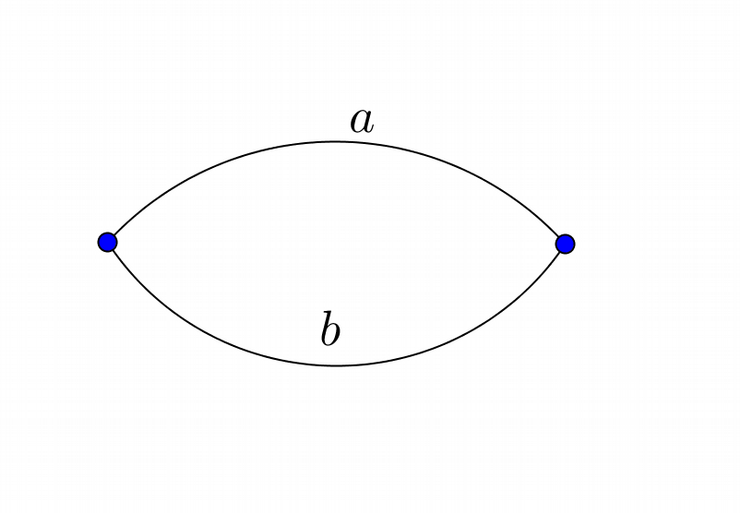}
        &\includegraphics[scale=0.3]{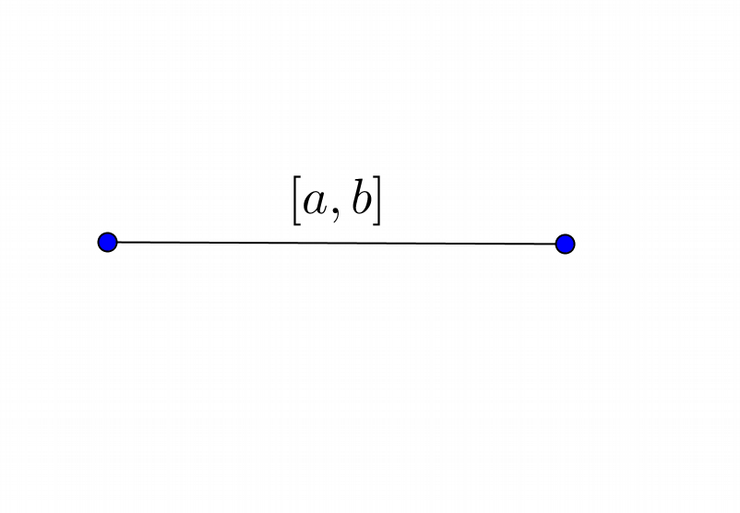}
\end{tabular}
\end{center}
\caption{Edge Collapse}
\label{fig:edgecollapse}
\end{figure}

\begin{lemma}
\label{lemma:connectivity}
Let $(G, A)$ be a simple, connected graph containing at least one vertex $v$.  Then $(G_v, A_v)$ is also connected.
\end{lemma}
\begin{proof} 
We consider two cases.

\begin{itemize}
\item
Case 1:  If $G_v$ is the empty graph or has only one vertex, then it is vacuously connected.

\item
Case 2:  Assume $G_{v}$ has 2 or more vertices.  Let $x$ and $y$ be distinct vertices of $G_v$.  Hence they are vertices in $G$ since $G_v$ doesn't introduce new vertices.  Since $G$ is connected, there is a path $P$ connecting $x$ and $y$ in $G$.  There are two cases.
\begin{itemize}
\item
First, assume $P$ does not pass through the vertex $v$.  The removal of $v$ does not affect $P$, so $P$ is also a path in $G_{v}$.  
\item
Now suppose $P$ passes through the vertex $v$. Then $P$ passes through vertices $p$ and $q$, where $p$ and $q$ are adjacent to $v$.  When we form $G_v$, the path $pvq$ in $G$ will be replaced by a single edge $pq$.  Thus, we can form a path $P_v$ from $x$ to $y$ in $G_v$ by replacing the edges $pv$ and $vq$ with the edge $pq$ in $G_{v}$.
\end{itemize}
\end{itemize}
\end{proof}

A graph can always be collapsed to a single edge, and even further to a single vertex.

\begin{lemma}
Let $(G, A)$ be a connected edge-weighted graph, and let $v$ and $w$ be vertices of $G$.  
\begin{enumerate}
    \item We can collapse $G$ to the graph $K_2$ with vertices $v$ and $w$.
    \item We can collapse $G$ to the vertex $v$.
\end{enumerate}
\end{lemma}
\label{theorem:reducetok2}

\begin{proof}
Choose a vertex other than $v$ and $w$ and do a star-clique operation.  This will reduce the number of vertices by one.  Now do an edge collapse to remove multiple edges.  Repeat until $v$ and $w$ are the only vertices left, and remove any multiple edges.  Since the graphs at each stage are connected, by Lemma \ref{lemma:connectivity}, we  are left with a single edge between $v$ and $w$.  

For part 2, a star-clique operation on $w$ will transform the edge $vw$ into the single vertex $v$.
\end{proof}

\section{The Spline Modules $S_G$ and $S_{G'}$}

Now that we have constructed the graph collapsing operations, we describe their effects on the spline module $S_G$. In order to do this, we must first fix an ordering of the vertices $v_1,\ldots, v_n$ of $G$.
Next, we partition $S_G$ into \textbf{flow-up classes} $\mathcal{F}_{i}$, which are subsets of splines with the first $i$ entries equal to zero and the $i+1$st entry non-zero.  A \textbf{flow-up class basis} will be a basis that consists of an element from each non-zero flow-up class.  

\begin{definition}
For $0 \leq i < n$, the {\textbf{$i$th flow-up class}}, denoted  $\mathcal{F}_i$,  is the set of splines of the form $F = (0, 0, \ldots, 0, f_{i+1}, \ldots, f_{n})$ where $f_{i+1}$ is non-zero. 
We  also  define $\mathcal{F}_{n} = \left\{\mathbf{0}\right\} =\mathbf{0} = (0,0,\ldots, 0)$, the zero spline.
\end{definition}

We now define the leading term of a spline to be its first non-zero entry.

\begin{definition}
Let $F \in \mathcal{F}_{i}$ where $i \neq n$, so that $F = (0, 0, \ldots, 0, f_{i+1}, \ldots, f_{n})$ and $f_{i+1} \ne 0$.  $L(F) = f_{i+1}$ is the \textbf{leading term} of $F$.
\end{definition}

\begin{definition}
A spline $B \in \mathcal{F}_i$ is a \textbf{minimal element} of $\mathcal{F}_i$ if $L(B)>0$, and for all $F \in \mathcal{F}_k$,  $(L(B) \leq  |L(F)|$.
\end{definition}

For example, $\mathbf{1} = (1, 1, \ldots, 1)$ has leading term $1$, which is minimal in $\N$.  Thus, $\mathbf{1}$ is a minimal element of $\mathcal{F}_{0}$.

{\textbf{Note:}} Although minimal elements of $\mathcal{F}_i$ have the same leading term, the other terms are not in general unique. 

Our strategy for constructing a basis for $S_G$ is to find a minimal element from each $\mathcal{F}_i$, for $i < n$.  The theorem below asserts that this set will alway form a basis, called a {\textbf{flow-up class basis}}.  

\begin{theorem}{\cite{liu}}
Let $(G,A)$ be an edge labeled graph with $n$ vertices. Let $\mathcal{B} = \left\{B_0, B_1,\dots, B_{n-1}\right\}\subset S_G$ where each $B_i$ is a minimal element of $\mathcal{F}_i$. Then $\mathcal{B}$ is a $\Z$-module basis for $S_G$. 
\end{theorem}

We now look at the effect of collapsing operations on flow-up classes and minimal elements.  We begin with an example.

\begin{example}
Let $G$ be the $Y$-graph with vertices $v_1,v_2,v_3,v_4$, and edge labels $a$, $b$, $c$, as in Figure \ref{fig:starclique}.  The star-clique operation on $v_{4}$ yields a $3$-cycle $G'$ with edge labels $(a,b)$, $(b,c)$, and $(a,c)$.  

Let  $(v_1,v_2,v_3,v_4) \in S_G$. The defining equations on $G$ are 
\begin{align}
\label{eqn:yd}
\begin{array}{rcl}
v_4     &\equiv &v_1 \bmod c\\
v_4   &\equiv &v_2 \bmod a\\ 
v_4     &\equiv &v_3 \bmod b
\end{array}
\end{align}

By the Theorem \ref{theorem:CRT}, Equation \ref{eqn:yd} has a solution  if and only if 
\begin{align}
\label{eqn:yd2}
\begin{array}{rcl}
v_1 &\equiv &v_2 \mod (a,c)\\
v_2 &\equiv &v_3 \mod (a,b)\\ 
v_3 &\equiv &v_1 \mod (b,c)
\end{array}
\end{align}

Note that these are in fact the defining equations of $S_ {G'}$.  Thus the projection map 
$$\phi:S_G \rightarrow S_{G'}$$ defined by  $\phi(v_1,v_2,v_3,v_4) = (v_1,v_2,v_3)$ is well defined, due to the forward direction of Theorem \ref{theorem:CRT}, and surjective by the reverse direction.  Since it's a projection map, it is also an $\Z$-module homomorphism.   Let $K$ be the kernel of $\phi$.  Then $\mathbf{g}$ is in  $K = \phi^{-1}(\mathbf{0}) $ if an only if $\mathbf{g} = (0,0,0,v_4)$ in $S_G$.  Substituting these values into (1), we get 
\begin{align}
\label{eqn:yd3}
\begin{array}{rcl}
v_4     &\equiv &0 \bmod a\\
v_{4}   &\equiv &0 \bmod b\\ 
v_4     &\equiv &0 \bmod c
\end{array}
\end{align}

which have as solutions $g_4 \equiv 0 \bmod [a,b,c]$.  Thus $\mathbf{g} = (0,0,0,[a,b,c])$ is a minimal element of $\mathcal{F}_3$ and forms a basis for $K$.
\end{example}

We can formalize this process for any graph.

\begin{theorem}
\label{theorem:surjective}
Let $(G,A)$ be an edge-weighted simple graph on vertices $\{v_1,\ldots, v_n\}$ and let $(G_n,A_n)$ be the result of a star-clique operation on vertex $v_n$.  The map $\phi:S_G \rightarrow S_{G_n}$ defined by $\phi(g_1,g_2,\ldots, g_n) = (g_1,g_2,\ldots ,g_{n-1})$ is a surjective $\Z$-module homomorphism with kernel  $\mathcal{F}_{n-1}\cup {\mathcal{F}}_{n} $.
\end{theorem}

\begin{proof}
The first thing to note is that the only edges affected by a star-clique operation are those incident to vertex $v$, say $v_{i_1},\ldots ,v_{i_d}$. Thus the defining equations for $S_G$ and $S_{G_n}$ only differ in that $S_G$ includes equations 
\begin{align*}
 g_n  &\equiv g_{i_1}\bmod a_{i_1}\\
 g_n  &\equiv g_{i_2} \bmod a_{i_2}\\
    &\vdots \\
 g_n  &\equiv g_{i_d} \bmod a_{i_d}
\end{align*}
whereas $S_{G_n}$ includes the equations 
$$g_{i_j}  \equiv g_{i_k} \bmod (a_{i_j},a_{i_k}), \text{ for all }  1 \leq j \leq d$$
By Theorem \ref{theorem:CRT}, the first set of equations implies that second, hence $\phi$ is well defined.  Also by Theorem \ref{theorem:CRT}, the second set of equations implies there exists a $g_n$ satisfying the first set, and since the remaining equations don't involve $g_n$, we have shown there exists a spline in $S_G$ mapping to a given spline in $S_{G_n}$.  

Now let $K$ denote the kernel of $\phi$.  Then $\mathbf{g} \in K = \phi^{-1}(\mathbf{0}) $ if and only if $\mathbf{g} = (0,\ldots,0,g_n)$ in $S_G$.  If $g_n \neq 0$, then $\mathbf{g}$ has exactly $n$ zeroes, hence is an element of $ {\mathcal{F}}_{n-1}$.  If $g_n = 0$, then $\mathbf{g} = \mathbf{0}$, the only element of $ {\mathcal{F}}_{n}$.
\end{proof}

\begin{theorem}
Let$(G,A)$ be an edge weighted graph and $(G',A')$ be the result of an edge collapse operation on multiple edges $e_1,e_2,\ldots, e_r$. Then $S_{G'} = S_{G}$.
\end{theorem}

\begin{proof}
Suppose that edges $e_1,e_2,\ldots, e_r$ each connect vertices $v$ and $w$, and have edge labels $a_1,a_2,\ldots, a_r$. The label for the new edge $e$ in $G',A'$ is $[a_1,a_2,\ldots, a_r]$. Thus the only difference between the defining equations for $S_G$ and $S_{G'}$ are the $r$ equations 
$g_v \equiv g_w \bmod a_i$ and the single equation $g_v \equiv g_w \bmod [a_1,a_2,\ldots, a_r]$ in $S_{G_e}$.  By Lemma \ref{lemma:gcd}, these two sets of equations have the same solutions, hence $S_{G'} = S_{G}$.
\end{proof}

In other words, collapsing multiple edges changes the graph, but not the corresponding spline module.  Putting these theorems together we now state our main result of this section.


\begin{theorem}
\label{theorem:reductionmap}
Let $(G',A')$ be the result of a sequence of collapsing operations on an edge-weighted graph $(G,A)$ resulting in the removal of vertices $v_{r+1},\dots v_n$ and the collapsing of any multiple edges.  Then the map $\phi:S_G \rightarrow S_{G'}$ defined by $\phi(g_1,g_2,\ldots, g_n) = (g_1,g_2,\ldots, g_r)$ is a surjective $\Z$-module homomorphism with kernel $\mathcal{F}_{r}\cup \cdots \cup \mathcal{F}_{n-1} \cup \mathcal{F}_{n}  $.
\end{theorem}

\begin{proof}
By Theorem \ref{theorem:surjective}, each time we do a star-clique operation, we have a surjective $\Z$- module homomorphism $\phi:S_G \rightarrow S_{G_v}$, and each time we do an edge collapse, the module stays the same.  Since $\phi$ is a composition of finitely many surjective $\Z$-module homomorphisms, it is also a surjective $\Z$-module homomorphism.

\

Let $K$ denote the kernel of $\phi$.  Then $\mathbf{g} \in K = \phi^{-1}(\mathbf{0}) $ if and only if $\mathbf{g} = (0,\ldots,0,g_{r+1},\ldots, g_n)$ in $S_G$.  Since $\mathbf{g}$ has at least $r$ zeroes, it is an element of $ \mathcal{F}_{i}$ for some $i \geq r$.
\end{proof}

\begin{corollary}

Let $(G,A)$ be an edge-weighted graph on $n$ vertices and $v_i$ and $v_j$ be distinct vertices of $G$.  The map $\phi:S_G \rightarrow S_{K_2}$ defined by $\phi(g_1,g_2,\ldots, g_n) = (g_i,g_j)$ is  a surjective $\Z$-module homomorphism.
\end{corollary}

\begin{proof}
We start with $(G,A)$ and delete vertices one by one until there are only two vertices left.  Each time we end up with multiple edges, we collapse them and replace $S_G$ with $S_{G_e}$. By Theorem \ref{theorem:reductionmap}, the result now follows. 
\end{proof}

\section{A basis for  $S_G$}
We now use the map $\phi$ to compare flow-up classes and minimal elements of $S_G$ and $S_{G'}$.

\begin{theorem}
\label{theorem:phitheorem} 
Let $(G',A')$ be the result of collapsing operations that remove vertices $v_{r+1},\ldots,v_n$.  Define $\phi: S_G \rightarrow S_{G'}$ by  $\phi(g_1,g_2,\ldots, g_n)=(g_1,g_2,\ldots, g_{r})$.   Suppose $i < r$, $F \in  {\mathcal{F}_{i}}$, and $K$ is the kernel of $\phi$. Then, 

\begin{enumerate}
\item $\phi({\mathcal{F}}) = {\mathcal{F}'_{i}}$
\item $L(F)=L(\phi(F))$.
\item $F$ is minimal if and only if $\phi(F)$ is minimal. 
\item Let $\mathcal{B}'$ and $\mathcal{K}$ be flow-up class bases for $S_{G'}$ and $K$, respectively.  Let $\mathcal{B}$ consist of a pre-image of each element of $\mathcal{B}'$.   Then $\mathcal{B}\cup\mathcal{K}$ is a flow-up class basis of $S_G$.
\item As $R$-modules, $S_G \cong S_{G'} \oplus K$.
\end{enumerate}
\end{theorem}

\begin{proof}

\ \\
\begin{enumerate}
\item  If $F = (0,\ldots,0,f_{i+1},\ldots, f_n)$ has $i$ leading zeroes, then $\phi(F) = (0,\ldots,0,f_{i+1},\dots f_r)$ also has $i$ leading zeroes.  
\item  $L(F) = f_{i+1} = L(\phi(F))$.
\item  If $F$ is minimal but $\phi(F)$ is not, then there is some $M'$ in ${\mathcal{F}'_{i}}$  such that $|L(M')| < |L(\phi{F})|$.  Let $M$ be in the inverse image of $M'$.  By (2), $L(M) = L(M')$, so  $|L(M)| < |L(\phi{F})|$, which contradicts the minimality of $F$. Conversely, suppose $F' = \phi(F)$ is minimal but $F$ is not.  Then there exists $M$ in ${\mathcal{F}_{i}}$ such that $|L(M)| < |L({F})|$. Then $L(M)=L(\phi(M))$, which implies $|L(\phi(M)| < |L({F'})|$, contradicting the minimality of $F'$.

\item Since $\mathcal{B}'$ and $\mathcal{K}$ are flow-up class bases, their leading terms are minimal in $S_{G'}$ and $K$, respectively.  Let $\mathcal{B}$ be a set of inverse images of elements of $\mathcal{B}'$.  By (3), these elements will minimal and there will be one from each flow-up class $\mathcal{F}_{0}, \ldots , \mathcal{F}_{r-1} $.  To complete a basis for $S_G$, we need a minimal element from each remaining non-zero flow-up classes $\mathcal{F}_{r}, \ldots , \mathcal{F}_{n-1}$.  By Theorem \ref{theorem:reductionmap}, $K=\mathcal{F}_{r}\cup \cdots \cup \mathcal{F}_{n-1} \cup \mathcal{F}_{n}$. Thus any flow-up class basis for $K$ will complete a basis of $S_G$, making  $\mathcal{B}\cup\mathcal{K}$ a flow-up class basis for $S_G$.
\item This follows from (4).
\end{enumerate}
\end{proof}

In fact, we can describe $S_G$ as the direct sum of rank 1 kernels of collapsing maps.  First, we define a complete collapse sequence of graphs.

\begin{definition} Let $(G,A)$ be a simple edge-weighted graph with $n$ vertices.
A {\textbf{complete collapse sequence}} is a set of graphs $\{G_n, \ldots, G_1\}$, where $G = G_{n}$ and for $2\leq i \leq n$,  $G_{i-1}$ is result of collapsing operations that remove the vertex $v_i$ from $G_{i}$.    
\end{definition}  

Note that each $G_i$ contains $i$ vertices and that $G_1$ is a graph consisting of the single vertex $v_1$.  

\begin{theorem} 
\label{thm:completecollapse}

Let $\{G_k\}$ be a complete collapse sequence of $(G,A)$.
Define $$S_{G_n} \xrightarrow{\phi_n} S_{G_{n-1}} \cdots \xrightarrow{\phi_3} S_{G_2}\xrightarrow{\phi_{2}} S_{G_1}\xrightarrow{\phi_{1}} {0}$$ 
by  $\phi_{i}(g_1,g_2,\ldots, g_{i})=(g_1,g_2,\ldots, g_{i-1})$ for $i>1$, and $\phi_1(g_1) = 0$.  
Let $K_i$ denote the kernel of $\phi_i$.  Then  

\begin{enumerate}
    \item $S_G \cong {K}_1 \oplus {K}_2 \oplus \cdots \oplus {{K}}_n$.
    \item $K_i$ is a rank 1 submodule of $S_{G_{i}}$.
    \item Let $st(v_{i})$ in $G_{i}$ have edge weights $\{c_1,\ldots, c_d\}$.  Then $m_1= (1)$ generates $K_1$ and and $m_i = (0,0,\ldots, [c_1,\ldots, c_d])$ generates $K_i$, for $i>1$.
    \item For each $i<n$, let $M_i$ in $S_G$ be a pre-image of $m_i$ under the map $\phi_n \circ \phi_{n-1} \cdots \circ \phi_{i+1}$, and let $M_n=m_n$. Then $\{M_1, \ldots M_n\}$ is a flow-up class basis for $S_G$ where $M_i \in {\mathcal{F}_{i-1}}$.
\end{enumerate}
 
\end{theorem}

\begin{proof}
By Theorem \ref{theorem:phitheorem}, for all $i>1$ $S_{G_{i}} \cong S_{G_{i-1}} \oplus \mathcal{K}_{i}$.  When $i=n$ and $n-1$, we have $S_{G} \cong S_{G_{n-1}} \oplus \mathcal{K}_{n}$ and $S_{G_{n-1}}  \cong S_{G_{n-2}} \oplus \mathcal{K}_{n-1}$.  Putting these together we have $$S_G = S_{G_{n-2}} \oplus \mathcal{K}_{n} \oplus \mathcal{K}_{n-1}$$  We can continue in this manner, down to $i=2$, to get $S_G \cong \mathcal{K}_n \oplus \mathcal{K}_{n-1} \oplus \cdots \oplus \mathcal{K}_{2} \oplus S_{G_1}$. The last map $\phi_1: S_{G_1} \rightarrow 0$ is the zero map, so $K_1 = S_{G_1}$, giving us the first result.  Since $G_1$ has no edges, and hence no defining equations,  $ S_{G_1}= R$.  Thus  $K_1 = R$ , which is generated by the element (1).  For $i>1$, elements of $K_i$ have form $(0,0,\ldots,0,g_{i})$.  Thus $K_i$ is isomorphic to a non-zero submodule of $R$,  hence $K_i$ has rank 1.  Elements of $K_i$  correspond to vertex labels $g_{i}$ for $v_{i}$ and 0 for all other vertices, so the defining equations  are $g_{i} \equiv 0 \bmod A(e_j)$ for each edge $e_j$ incident to $v_{i}$,  and  $0 \equiv 0 \bmod A(e)$ for all other edges.  The incident edges are precisely the edges in $st(v_{i})$, which have weights $\{c_1,\ldots, c_d\}$.  Then $(0,0,\ldots,0,g_{i}) \in K_i$ if and only if  
$g_{i} \equiv 0 \bmod c_j$ for all $j$.  But by Lemma \ref{lemma:gcd}, this is equivalent to $g_{i} \equiv 0 \bmod [c_1,\ldots, c_d]$.  So $m_i = (0,0,\ldots, [c_1,\ldots, c_d])$ is a minimal element of $K_i$, which means $m_i$  generates $K_i$.

Finally, since $m_i$ is a minimal element of of the $i-1$st flow-up class of $S_{G_{i}}$, any pre-image $M_i$ in $S_G$ will be a minimal element of ${\mathcal{F}}_{i-1}$, by Theorem \ref{theorem:phitheorem}. Thus the $M_i$'s form a flow-up class basis of $S_G$.  
\end{proof}

This gives a procedure for constructing a basis $\{M_1, \ldots M_n\}$ for $S_G$.  For $i>1$,  $M_i$ has leading term $[c_1,\ldots, c_d]$ where $\{c_1,\ldots, c_d\}$ are weights of edges incident to vertex $v_{i}$ in $G_{i}$, and we can use the Chinese Remainder Theorem to explicitly construct the other terms of  $M_i$. When $i=1$, $G_1$ is a single vertex with no edges.  Since there are no restrictions, $m_1=(1)$, and it's trivial to see that $M_1=(1,\ldots,1)$ is always a spline in $S_G$.

\begin{figure}
\begin{center}
\begin{tabular}{cc}
\includegraphics[scale=0.3]{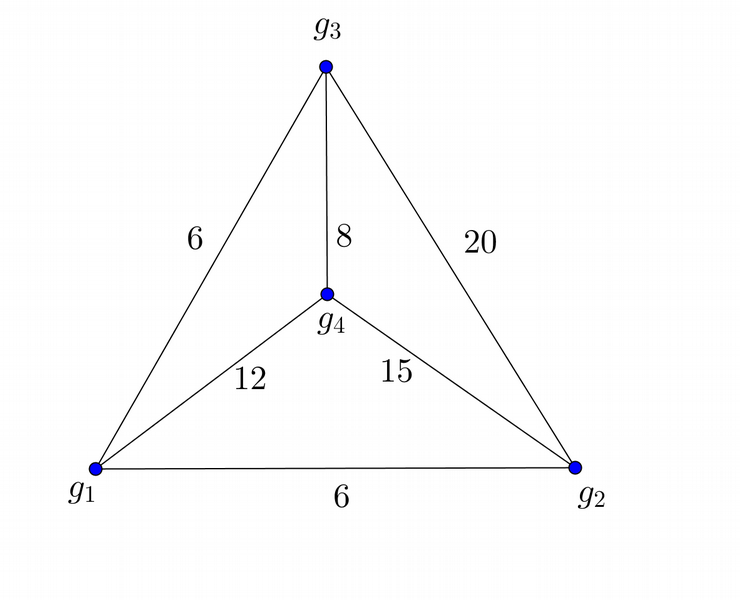}
    &\includegraphics[scale=0.3]{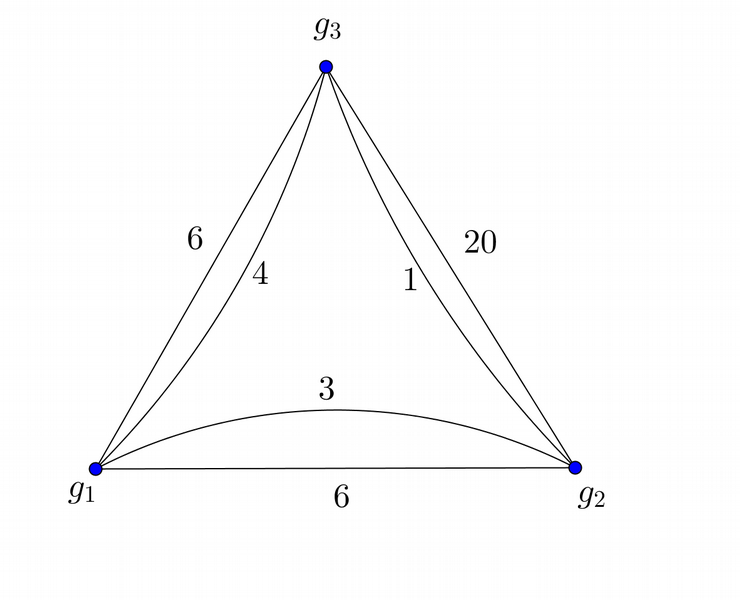}\\
Figure \ref{fig:basis}a
    &Figure \ref{fig:basis}b\\
\includegraphics[scale=0.3]{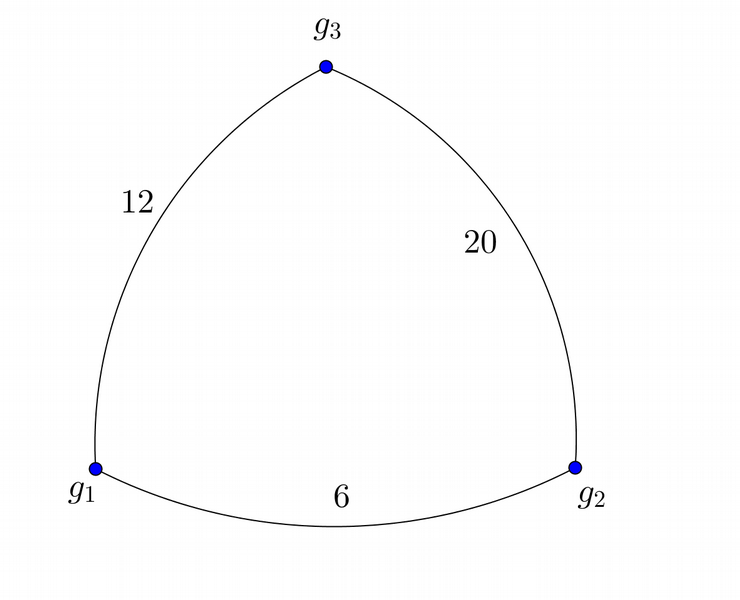}
    &\includegraphics[scale=0.3]{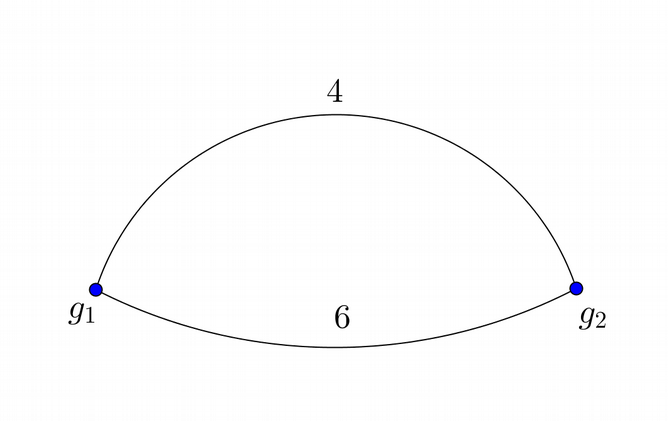}\\
Figure \ref{fig:basis}c
    &Figure \ref{fig:basis}d
\end{tabular}

\begin{tabular}{c}
\includegraphics[scale=0.3]{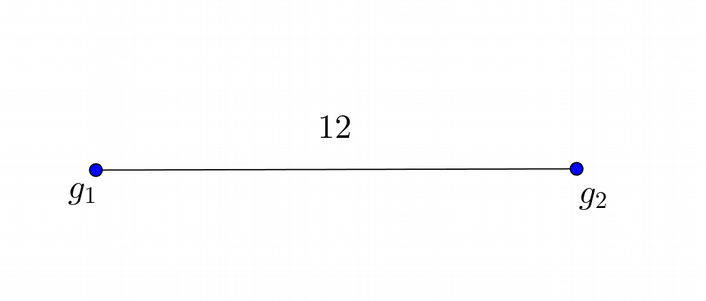}\\
Figure \ref{fig:basis}e
\end{tabular}
\caption{Basis for $K_{4}$}
\label{fig:basis}
\end{center}
\end{figure}

\begin{example}
\label{example:k4}
Let $G = K_{4}$ with the sequence of collapsing operations in Figures \ref{fig:basis}a through \ref{fig:basis}e.  

In order to find $M_4$, we set $g_{1} = 0$, $g_{2} = 0$, $g_{3} = 0$ in Figure \ref{fig:basis}a.  
This gives us the follow equations.
\begin{align*}
g_{4}   &\equiv 0 \bmod 12\\
g_{4}   &\equiv 0 \bmod 15\\
g_{4}   &\equiv 0 \bmod 8
\end{align*}
By Theorem \ref{thm:completecollapse}, the leading term $g_4$ will be the LCM of the edge weights incident to $v_4$, so   $g_{4} =[12, 15, 8] = 120$.  This gives basis element $M_{4} = (0, 0, 0, 120)$.

Next let $g_{1} = 0$, $g_{2} = 0$ in (Figure \ref{fig:basis}c).  The equations are
\begin{align*}
g_{3}   &\equiv 0 \bmod 20\\
g_{3}   &\equiv 0 \bmod 12
\end{align*}
which has least solution $g_{3} =[12, 20] = 60$.  We can let $g_{4} = 0$, which gives us a basis element $M_{3} = (0, 0, 60, 0)$. 

Now consider Figure \ref{fig:basis}e and let $g_{1} = 0$. Then $g_{2} \equiv 0 \bmod 12$, so $g_{2} = 12$ is minimal.  We are guaranteed values for $g_3$ and $g_4$ by the Chinese Remainder Theorem, but in this case we easily find a basis element $M_{2} = (0, 12, 12, 12)$.

Finally, we can always set $M_{1} = \left(1,1,1,1\right)$.

\end{example}

In the next section we find an alternate construction of $L(M_i)$, in terms of the original edge weights of $G$.  This construction does not require collapsing $G$ to the graph $G_i$, in terms of the original edge weights of $G$.

\section{Leading Terms of Basis Elements}

Let $(G,A)$ be an edge-weighted graph and let  $(G', A')$ be the result of a star-clique operation removing vertex $v$ from $G$.  Let $w$ and $u$ be vertices in $G$ distinct from $v$.  We introduce the following notation.

\begin{itemize} 
\item $(p)$ = the GCD of the edge labels of a path $p$.
\item  $\mathcal{P}$ = the set of paths between  $w$ and $u$  in $G$.
\item  ${\mathcal{Q}}$ = the set of paths between $w$ and $u$ in $G'$.
\item $[\mathcal{P}]$ = the LCM of the $(p)$'s, where $p \in \mathcal{P}$.
\item $[\mathcal{Q}]$ = the LCM of the $(q)$'s, where $q \in \mathcal{Q}$.
\end{itemize}

Next, we show that  $[\mathcal{P}]$ is invariant under the star-clique operation.

\begin{lemma}
\label{lemma:LCM invariance 1}
Let $v$ be a vertex of  $G$ and let $(G', A')$ be the result a star-clique operation.  Let $w, u$ be vertices in $G$ distinct from $v$.  Then $[\mathcal{P}] =  [\mathcal{Q}]$.
\end{lemma}
\begin{proof}
 We partition $\mathcal{P}$ into two sets: 
\begin{itemize}
\item  $\mathcal{P}_+$ =  paths that include $v$;
\item ${\mathcal{P}}_-$ = paths not including $v$.
\end{itemize}

Let $p$ be a path in $\mathcal{P}_-$.  Since this path doesn't includes $v$, it will also be a path in $G'$, hence $p$ is in $\mathcal{Q}$.   Now let $p$ be a path in $\mathcal{P}_+$ with edge labels $c_{1}, c_{2},  \ldots, c_{n}$. Since this $p$ includes $v$, it contains two vertices $x$ and $y$ and two edges $e$ and $f$  adjacent to $v$,  with edge labels $c_i$ and $c_{i+1}$.  Thus,
$$(p) = (c_{1}, c_{2}, \ldots, c_i, c_{i+1}, \ldots, c_{n})$$
However, by the star-clique construction, there will be an edge in $G'$  directly between  $x$ and $y$  with edge label  $(c_i, c_{i+1})$. Consequently there is a path $p'$ in $G'$ with
$$(p') = (c_{1}, c_{2}, \ldots, (c_{i}, c_{i+1}), \ldots, c_{n}).$$
But since $(a, (b, c)) = (a, b, c)$, we  have $(p) = (p')$.

Now let $q = e_1\cdots e_n$ be a path in $\mathcal{Q}$ and $$(q) = (c_{1}, c_{2}, \ldots,  c_{n})$$ If an $e_i$ was not an edge in $G$, then its vertices $x$ and $y$  must be adjacent to $v$ in $G$.  Thus there are edges $e_x$ and $e_y$ in $G$ with weights $c_x$ and $c_y$, which means  $e_i$ in $G'$ has weight $c_i =(c_x,c_y)$.  Let $q_i = e_1\cdots e_{i-1}e_xe_ye_{i+1} e_n$, then 
$$(q_i) = (c_{1}, c_{2}, \ldots, c_{i-1},c_x,c_y, c_{i+1}, \ldots, c_{n})$$
$$ = (c_{1}, c_{2}, \ldots, c_{i-1},(c_x,c_y), c_{i+1}, \ldots, c_{n})$$
$$ = (c_{1}, c_{2}, \ldots, c_{i-1},c_i, c_{i+1}, \ldots, c_{n}) = (q)$$
Since we can do this for each such edge $e_i$, we see that there is a path $q^+$ in $G$ with  $(q^+) = (q)$.  Consequently, the GCDs used to compute $[\mathcal{P}]$ are identical to the GCDs used to compute $[\mathcal{Q}]$, so $[\mathcal{P}]=[\mathcal{Q}]$.
\end{proof}

We now prove that $[\mathcal{P}]$ is invariant under edge-collapse.
\begin{lemma}
\label{lemma:LCM invariance 2}
Let $(G', A')$ be the result an edge-collapse operation, and let $w, u$ be vertices of $G$.  Then $[\mathcal{P}] =  [\mathcal{Q}]$.
\end{lemma}

\begin{proof}
Let $d_1,d_2,\ldots,d_r$ with weights  $c_1c_2,\ldots, c_r$ be the edges in $G$ that are collapsed to an edge $d$ with weight $c=[c_1,c_2,\ldots,c_r]$ in $G'$.   

We can partition $\mathcal{P}$  and  $\mathcal{Q}$ into two sets:
\begin{itemize}
\item $\mathcal{P}_1$ = paths that include any of the $d_i$
\item $\mathcal{P}_2$ = paths that include none of the $d_i$
\item $\mathcal{Q}_1$ = paths that include $d$
\item $\mathcal{Q}_2$ = paths that don't include $d$
\end{itemize}

Let $p \in\mathcal{P}_2$.  Since $p$ does not include any $d_i$, all of its edges are in $G'$, so $p \in \mathcal{Q}_2.$  Similary, if   $p \in \mathcal{Q}_2$ then $p \in \mathcal{P}_2.$ Thus   $\mathcal{P}_2 = \mathcal{Q}_2$.   

Now let $p_1 \in \mathcal{P}_1$ contain edges $d_1$, with remaining edges $e_1,\ldots, e_s$ in both $G$ and $G'$, and edge weights $a_1,\ldots, a_s$.   Clearly, for each $1 < i \leq r$ there is a path $p_i \in \mathcal{P}_1$ with edges $d_i, e_1,\ldots, e_s $.  Thus we have
\begin{align*}
(p_1) &= (c_i, a_{1}, a_{2}, \ldots, a_{s})\\
(p_2) &= (c_i, a_{1}, a_{2}, \ldots, a_{s})\\
&\vdots \\
(p_r) &= (c_i, a_{1}, a_{2}, \ldots, a_{s})
\end{align*}

These correspond directly to the path $p'$ in $\mathcal{Q}_1$ with edges $d, e_1,\ldots, e_s$
and  $$(p') = ([c_1,c_2,\ldots,c_r], a_{1}, a_{2}, \ldots, a_{s})$$
Similarly, any $p' \in  \mathcal{Q}_1$ corresponds to $p_1,\ldots,p_r$ in $\mathcal{P_1}$.
Let $a = (a_{1}, a_{2}, \ldots a_{n})$.    By repeated applications of Lemma \ref{lemma:gcd}, we have 
\begin{align*}
[(p_1),\dots (p_r)]
&=[(c_1, a_{1}, a_{2}, \ldots a_{s}), \ldots, (c_r, a_{1}, a_{2}, \ldots a_{s})]\\ 
&= [(c_1, (a_{1}, a_{2}, \ldots a_{s})),\ldots a_{s})), \ldots, (c_r, (a_{1}, a_{2}, \ldots a_{s}))]\\ 
&= [(c_1, a), \ldots, (c_r, a)]\\ 
&= [(a, [c_1, \ldots, c_r )] \\
&= (p')
\end{align*}

Let $\mathcal{Q}_1 = \{x_1,\dots x_m\}$, $\mathcal{Q}_2 = \{y_1,\dots, y_k\}$ and $y = [y_1,\dots y_k].$  Then  each $x_i \in \mathcal{Q_1}$ corresponds to $r$ different paths $x_{ij}  \in \mathcal{P}_1$, and by above, each$(x_i) =[(x_{i1}),\dots (x_{ir})]$.

By repeated applications of  Lemma \ref{lemma:gcd} we have 

\begin{align*}
[\mathcal{Q}] 
&= [(x_1),\dots, (x_m), (y_1),\dots, (y_k)]\\
&=  [(x_1),\dots, (x_m), [(y_1),\dots, (y_k)]\\
&=  [(x_1),\dots, (x_m), y]\\
&=  [(x_{11}),\dots, (x_{1r})],\dots, [(x_{m1}),\dots, (x_{mr})], y]\\
&= [(x_{11}),\dots, (x_{1r})],\dots, [(x_{m1}),\dots, (x_{mr})], [(x_1),\dots, (x_m), (y_1),\dots, (y_k)] ]\\
&= [(x_{11}),\dots, (x_{1r}),\dots, (x_{m1}),\dots (x_{mr}), (x_1),\dots, (x_m), (y_1),\dots, (y_k) ]\\
&= [\mathcal{P}]
\end{align*}

\end{proof}

We have the following important corollary. 

\begin{corollary}
\label{corollary:lcmgcd}
Let $w,u$ be two vertices in both $G$ and $G'$.  The LCM of the GCD of all paths from $w$ to $u$ is invariant under collapsing operations taking $G$ to $G'$.
\end{corollary}

Using this result, we can construct leading terms of basis elements in terms of LCMs and GCDs of the edge weights of $G$. 

\begin{theorem}
\label{theo:LCM of GCD}
Let $M_1,M_2,\ldots M_{n}$ be a flow up class basis for $S_G$ and let $\mathcal{P}_{ij}$ denote the set of paths from $v_i$ to $v_j$ in $S_G$ where $i<j$.  Then $L(M_1) = 1$, and for $i>1$, and $L(M_{i}) = [ [\mathcal{P}_{i1}],[\mathcal{P}_{i2}],\ldots, [\mathcal{P}_{i,i-1}] ]$.
\end{theorem}

\begin{proof}
$M_1 = (1,\ldots, 1)$ is minimal in ${\mathcal{F}}_0$ since 1 is minimal in $\mathbb{N}$.  We know from Theorem  that $L(M_{i}) = L(m_i)$ where $m_i = (0,0,\ldots, [c_1,\ldots, c_r])$ is a minimal element of $K_i$, and $c_j$ is the edge label of the edge between $v_i$ and $v_j$ , for all $ j  \leq r \leq i$.  (If necessary we can temporarily reorder the $ v_j$'s incident to $v_i$ so that they are vertices $v_1,\ldots,v_r$.) For each $j$,  the edge $e_j$ with label $c_j$  is the only path form $v_i$ to $v_j$ in $G_i$.     By Lemma \ref{lemma:gcd},  this means that  $c_j =[\mathcal{P}_{ij}]$.  Then  $L(M_{i}) = [c_1,\ldots, c_r]= [ [\mathcal{P}_{i1}],[\mathcal{P}_{i2}],\ldots, [\mathcal{P}_{i,r}] ]$.  If $r=i-1$, we are done.  If $r<i$, let $p$ be a path  from $v_i$ to  $v_k$  where $ r  < k \leq i$.  This path necessarily passes through some $v_j$ for some $ j\leq r$, so $d_j$ is one of the edges of $p$.  Let $b$ be the GCD of the edge weights of the remaining edges of $p$.  Then $(p) = (c_j, b)$.   Since $(p)$ divides $c_j$, we have  $[c_j] = [c_j,(p)]$ and thus $[c_1,\ldots, c_r] = [c_1,\ldots, c_r, (p)]$.
Since we can do this for any path from $v_i$ to $v_k$,  we get that $[c_1,\ldots, c_r] = [c_1,\ldots, c_r, [P_{ik}]]$ and by varying $k$ we get that $L(M_{i}) = [c_1,\ldots, c_r,  [\mathcal{P}_{i,r}],\ldots, [\mathcal{P}_{i,i-1}]]= [ [\mathcal{P}_{i1}],[\mathcal{P}_{i2}],\ldots, [\mathcal{P}_{i,i-1}] ]$.
\end{proof}

We see from this proof that we don't need to check all paths from $v_i$ to all vertices $v_j$, just paths from $v_i$ to $v_j$ where all interior vertices have subscripts bigger than $i$.  An easy way to identify these paths is to label all vertices $v_{1}$ to $v_{i-1}$ with 0's, and use only paths from $v_i$ to the first 0 you run into. 

\begin{figure}
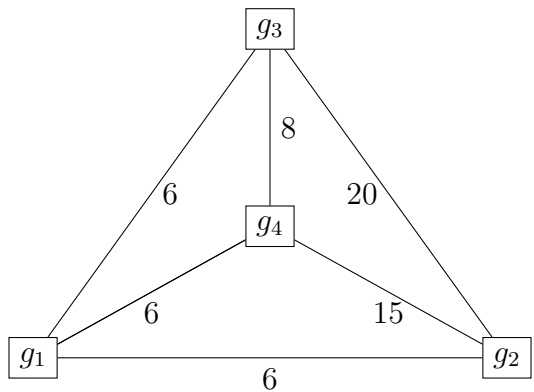

\begin{center}
\begin{tabular}{cc}
$\kfour{0.35}{g_{1}}{g_{2}}{g_{3}}{g_{4}}$
\end{tabular}
\end{center}
\caption{Edge-weighted $K_{4}$}
\label{fig:k4 graph}
\end{figure}



\begin{figure}
\begin{center}
\begin{tabular}{cc}
\raisebox{1in}{Figure \ref{fig:nocollapse}a}
    &\includegraphics[scale=0.35]{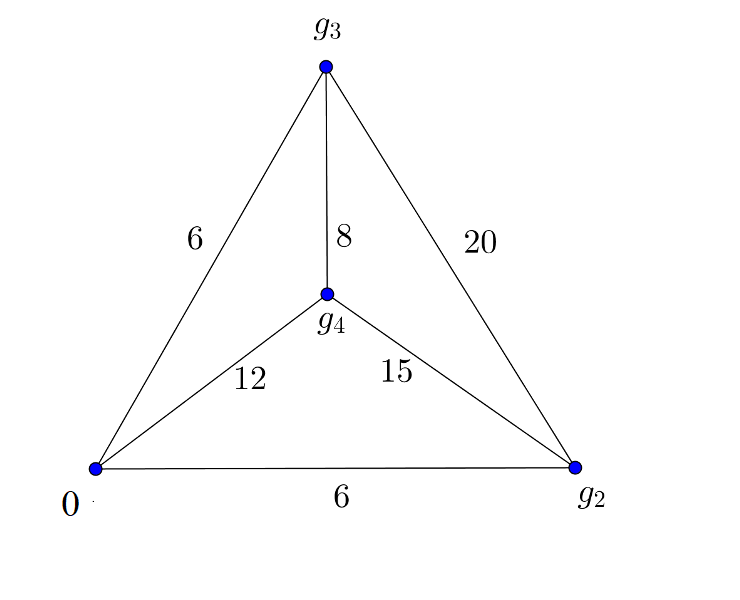}\\
\raisebox{1in}{Figure \ref{fig:nocollapse}b}
    &\includegraphics[scale=0.35]{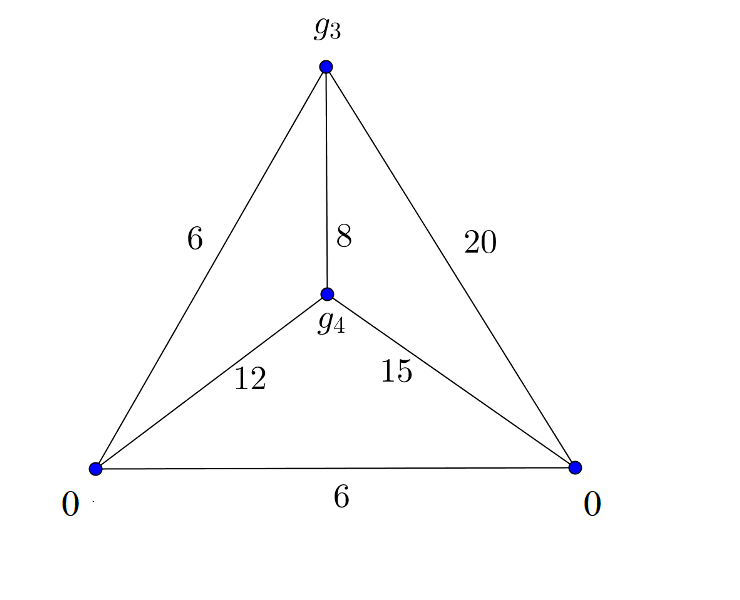}\\
\raisebox{1in}{Figure \ref{fig:nocollapse}c}
    &\includegraphics[scale=0.35]{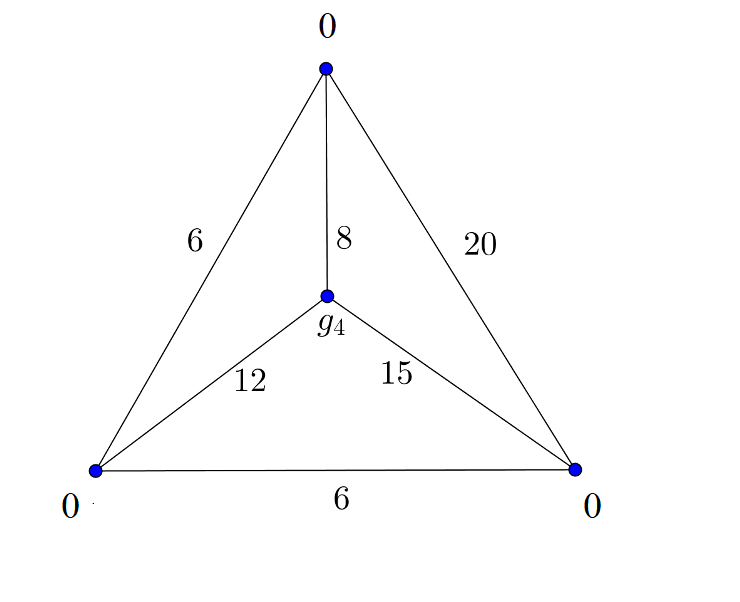}
\end{tabular}
\end{center}
\caption{Finding Leading Terms Using Theorem \ref{theo:LCM of GCD}}
\label{fig:nocollapse}
\end{figure}

\begin{example}
We will use Theorem \ref{theo:LCM of GCD} to compute the leading terms of a flow class basis for $S_{G}$, the $K_{4}$ graph of Figure \ref{fig:k4 graph}.

\begin{enumerate}
\item First, we have $L(M_1)=1$. 
\item In order to find $L(M_{2})$ we set $g_1=0$ as in Figure \ref{fig:nocollapse}a. Then $L(M_{2})$ is the LCM of the GCD of all paths from $g_{2}$ to any vertex labeled $0$.  These paths are:
\begin{itemize}
\item
$g_{2}g_{1}$, which has edge weight $6$,
\item
$g_{2}g_{4}g_{1}$, which has edge weights $12, 15$ 
\item
$g_{2} g_{3}g_{1}$, which has edge weights $20, 6$ 
\item
$g_{2}g_{4}g_{3}g_{1}$, which has edge weights $15, 8, 6$ 
\item
$g_{3}g_{4}g_{1}$, which has edge weights $20, 8, 12$ 
\end{itemize}
Thus  $L(M_{2}) = [6,(12,15),(20,6),(15,8,6),(20,8,12)]=[6, 3, 2, 1, 4] = 12$.

\item For  $L(M_{3})$, we set both $g_{1}$, $g_{2}$ equal $0$  as in Figure \ref{fig:nocollapse}b.
The paths from $g_{3}$ to any vertex labeled $0$ are:
\begin{itemize}
\item
$g_{3}g_{2}$, with edge weight $20$,
\item
$g_{3}g_{1}$, with edge weight $6$,
\item
$g_{3}g_{4}g_{1}$, with edge weights $8, 12$ 
\item
$g_{3}g_{4}g_{2}$, with edge weights $8, 15$
\end{itemize}
Thus $L(M_{3}) = [20, 6, (8,12), (8,15)]=[20, 6, 4, 1] = 60$.

\item Finally to compute $L(M_{4})$ we set $g_{1}$, $g_{2}$, $g_{3}$ equal to $0$, as in Figure \ref{fig:nocollapse}c, and find the paths from $g_{4}$ to $0$:
\begin{itemize}
\item
$g_{4}g_{1}$, with weight $12$,
\item
$g_{4}g_{2}$, with weight $15$,
\item
$g_{4}g_{3}$, with weight $8$,
\end{itemize}

Then $L(M_{4})= [12, 15, 8] = 120$.
\end{enumerate}
\end{example}

Thus the leading terms of $M_1$, $M_2$, $M_3$, $M_4$ are $1$, $12$, $60$, and $120$ respectively.  These agree with the leading terms found in Example \ref{example:k4}, illustrating that Theorems \ref{thm:completecollapse} and \ref{theo:LCM of GCD} provide two different ways to contruct the same basis for $S_G$. 

\bibliographystyle{amsplain}

\end{document}